\documentclass[12p]{amsart}
\usepackage{amsmath,amsfonts,amsthm,latexsym}
\usepackage{enumerate,etoolbox}
\usepackage[colorlinks=true, urlcolor=blue, linkcolor=blue, citecolor=blue]{hyperref}
\numberwithin{equation}{section}
\newtheorem{thm}{Theorem}
\numberwithin{thm}{section}

\newtheorem{dfn}[thm]{Definition}
\newtheorem{conj}[thm]{Conjecture}
\newtheorem{prop}[thm]{Proposition}
\newtheorem{lemma}[thm]{Lemma}
\newtheorem{cor}[thm]{Corollary}
\newtheorem{rmk}[thm]{Remark}
\newtheorem{exm}[thm]{Example}

\author{Hwangrae Lee}
\address{Department of Mathematics, Pohang University of Science and Technology, Korea \\\url{meso@postech.ac.kr}}
\title{The Euclidean Distance Degree \mbox{of Fermat Hypersurfaces}}
\begin{document}

\maketitle

\begin{abstract}
Finding the point in an algebraic variety that is closest to a given point is an optimization problem with many applications. We study the case when the variety is a Fermat hypersurface. Our formula for its Euclidean distance degree is a piecewise polynomial whose pieces are defined by subtle congruence conditons.
\end{abstract}

\section{Introduction} 
Let $X \subset \mathbb{R}^n$ be an real affine algebraic variety, i.e.~$X$ is the common zero set of some polynomials $f_1,\dots,f_m \in \mathbb{R}[x_1,\dots,x_n]$. We consider the following problem: given $u \in \mathbb{R}^n$, compute $u^* \in X$ that minimizes the squared Euclidean distance $d_u{x} = \sum_{i=1}^n (u_i-x_i)^2$ from the given point $u$. This problem arises from best approximation problems. Once we have a mathematical model $X$ to be satisfied by a data $u$ obtained by, for example, an experiment or reception from someone's transmission, usually $u$ contains some error and hence we want to correct it. The nearest point $u^*$ in $X$ to $u$ represents the original data suggested by $u$.

In order to find $u^*$ algebraically, we consider the zeroes in $\mathbb{C}^n$ of the equations defining $X$, and we examine all {\em complex} critical points of the squared distance function $d_u(x) = \sum_{i=1}^n (u_i-x_i)^2 $ on $X \setminus X_{sing}$ where $X_{sing}$ is the singular locus of $X$. If $X$ has some singular locus, then there could be infinitely many critical points of $d_u(x)$ on $X$. Thus we remove the singular locus of $X$. The number of such critical points is finite and constant on a dense open subset of data $u\in \mathbb{R}^n$. That number of critical points was studied by J.Draisma et al \cite{DHOST}. It is called the Euclidean distance degree (ED-degree) of the variety $X$, and denoted as $EDdeg(X)$. From now on, all the objects will be considered as complex varieties, except in Section~2.3.

Sometimes, $X$ is given by homogeneous polynomials. The set of $m_1$ by $m_2$ matrices of rank at most $k$ is a typical example. Such a variety is called a {\em projective} algebraic variety in $\mathbb{P}^n(\mathbb{C})$. For the definition of $\mathbb{P}^n(\mathbb{C})$ and more informations, see Chapter 8 of the book by Cox, Little, and O'Shea \cite{CLO}. For a projective $X \subset \mathbb{P}^n(\mathbb{C})$, we define $EDdeg(X)$ to be the ED-degree of the affine cone of $X$ in $\mathbb{C}^{n+1}$. That is, just regard $X$ as an affine variety and compute the ED-degree. The ED-degrees of determinantal varieties as above have been studied by G.Ottaviani et al \cite{OSS}.

This paper is motivated by following general upper bound on the ED-degree.
\begin{prop} \cite[Corollary 2.9]{DHOST} \label{gbound}
Let $X$ be a hypersurface in $\mathbb{P}^n(\mathbb{C})$ defined by a homogeneous polynomial $f$ of degree $d$. Then
$$
EDdeg(X) \leq d\sum_{i=0}^{n-1} (d-1)^i.
$$
and equality holds when $f$ is generic.
\end{prop}
In this paper, we focus on Fermat hypersufaces and their variations.
\begin{dfn}
\begin{itemize}
\item[]
\item A Fermat hypersurface of degree $d$ in $\mathbb{P}^n(\mathbb{C})$, denoted by $F_{n,d}$ is the projective variety defined by the polynomial $x_0^d + \cdots + x_n^d$.
\item An affine Fermat hypersurface of degree $d$ in $\mathbb{C}^n$, denoted by $AF_{n,d}$ is the affine variety defined by the polynomial $x_1^d + \cdots + x_n^d - 1$.
\item A scaled Fermat hypersurface of degree $d$ in $\mathbb{P}^n(\mathbb{C})$ with scaling vector $a=(a_0,\dots,a_n) \in (\mathbb{C}^*)^{n+1}$, denoted by $SF_{n,d}^a$ is the projective variety defined by the polynomial $x_0^d/{a_0} + \cdots + x_n^d/{a_n}$.
\end{itemize}
\end{dfn}
In statistical optimization, maximum likelihood estimation (MLE) is an important tool. The generic number of the critical points of maximum likelihood function, called ML-degree, is a parallel concept to ED-degree. The ML-degrees of many statistically relevant varieties have been computed \cite{SH}. Recently, in particular, the ML-degree of $F_{n,d}$ is partially given by D.Agostini et al \cite{mldeg}. Their results, which we review in Example \ref{curve}, serve as motivation our study of the ED-degree of $F_{n,d}$.

This paper is organized as follows. In Section 2, we will investigate the sharpness of the general bound (Proposition \ref{gbound}) for the Fermat hypersurfaces. We give a formula for the ED-degree of $F_{n,d}$ (Theorem \ref{main}), and gives an explicit formula for $n\leq 3$ (Remark \ref{delta3}, Example \ref{curve}). If we fix $n$ and consider the general bound as a function in $d$, it is the best possible polynomial bound (Lemma \ref{sharp}), while the gap can be arbitrary large (Remark \ref{delta3}). The main theorem can be used for an efficient algorithm which computes the ED-degree of Fermat hypersurfaces numerically (Example \ref{m2}). The proof of Theorem \ref{main} can be used similarly to evaluate the ED-degree for an {\em affine} Fermat hypersurface $AF_{n,d}$ (Corollary \ref{affcase}). After that, an open problem (Conjecture \ref{con}) about {\em real} Fermat hypersurfaces will be discussed.

In Section 3, we will consider the scaled Fermat hypersurfaces for fixed $n$ and $d$. We introduce the {\em exponential cyclotomic polynomial} $Q_{m,p}$ which has a special role for the scaling vector $a$ of $SF_{n,d}^a$ (Theorem \ref{resultant}). As a corollary,  we will see that the ED-degree of scaled Fermat hypersurface usually achieves the general bound.

\section*{Acknowledgements}
The author would like to thank his advisors Bernd Sturmfels and Hyungju Park for their guidance, comments and support, and Donghoon Hyeon for useful conversations and suggestions. Discussions about Section 3 with Donggeon Yhee were also helpful. The author was partially supported by POSTECH-IBS and NIMS-CAMP.

\section{ED-degree for Fermat hypersurfaces}
\subsection{Main theorem for Fermat hypersurfaces}
In this section, we compute the ED-degree of $F_{n,d}$ for each $n,d$.
\begin{dfn}
For a positive integer $p$, fix a $p$-th primitive root of unity $\zeta$. Define $\delta(m,p)$ to be the number of integer $m$-tuples $(t_1,\dots,t_m)$, $1 \leq t_i \leq p$, satisfying 
$$1+\sum_{i=1}^m \zeta^{2t_i} = 0.$$
Note that it does not depend the choice of $\zeta$.
\end{dfn}
\begin{thm} \label{main}
The ED-degree of the Fermat hypersurface $F_{n,d}$ is given by
$$EDdeg(F_{n,d}) = d\sum_{i=0}^{n-1} (d-1)^i - \sum_{m=1}^{n} \binom{n+1}{m+1} \cdot \delta(m,d-2)$$
\end{thm}

\begin{rmk} \label{delta3}
For small $m$, the following are derived easily from the definition:

\begin{enumerate}[(i)]
\item $$\delta(1,p) =
\left\{ \begin{array}{ll}
2 & \text{if } p \equiv 0 \text{ mod } 4\\
0 & otherwise
\end{array}
\right.
$$
\item $$
\delta(2,p) =
\left\{ \begin{array}{ll}
8 & \text{if } p \equiv 0 \text{ mod } 6\\
2 & \text{if } p \equiv 3 \text{ mod } 6\\
0 & otherwise
\end{array}
\right.
$$
\item $$\delta(3,p) =
\left\{ \begin{array}{ll}
12p-24 & \text{if } p \equiv 0 \text{ mod } 4\\
0 & otherwise
\end{array}
\right.
$$
\end{enumerate}
In particular, (iii) implies that the difference between the general bound and ED-degree can be arbitrary large. Although, following lemma shows that the general bound is the best possible polynomial bound.
\end{rmk}

\begin{lemma} \label{sharp}
If p is a prime bigger than $m+1$, then $\delta(m,p)=0$.
\end{lemma}
\begin{proof}
Assume $1+\sum_{i=1}^m \zeta^{2t_i} = 0$. Replacing $2t_i$ by $2t_i-p$ if $2t_i \geq p$ for each $i$, we have a polynomial in $\zeta$ whose degree is less than $p$. Since $p$ is a prime, it should be a scalar multiple of the cyclotomic polynomial $\Phi_p(\zeta):=1+\zeta+\cdots+\zeta^{p-1}$. It has $p$ terms, hence $m+1\geq p$. 
\end{proof}

No closed formula for $\delta(m,p)$ is known, but it has been studied in both algebraic geometry and number theory \cite{betti,Lau,Mann}. In particular, Theorem 2 in \cite{Lau} implies that $\delta(m,p)$ is a polynomial periodic function in $p$. 

\begin{cor}
For fixed $n$, the ED-degree of $F_{n,d}$ is a polynomial periodic function in $d$.
\end{cor}

\begin{exm} \label{curve}
In \cite{mldeg}, the ML-degree of the Fermat curves $(n=2)$ is given by
$$
MLdeg(F_{2,d}) =
\left\{ \begin{array}{ll}
d^2+d & \text{if } d \equiv 0,2 \text{ mod } 6\\
d^2+d-3 & \text{if } d \equiv 3,5 \text{ mod } 6\\
d^2+d-2 & \text{if } d \equiv 4 \text{ mod } 6\\
d^2+d-5 & \text{if } d \equiv 1 \text{ mod } 6.
\end{array}
\right.
$$
By the Theorem \ref{main}, we have
$$
EDdeg(F_{2,d})=
\left\{ \begin{array}{ll}
d^2 & \text{if } d \equiv 0,1,3,4,7,9 \text{ mod } 12\\
d^2-2 & \text{if } d \equiv 5,11 \text{ mod } 12\\
d^2-6 & \text{if } d \equiv 6,10 \text{ mod } 12\\
d^2-8 & \text{if } d \equiv 8 \text{ mod } 12\\
d^2-14 & \text{if } d \equiv 2 \text{ mod } 12.
\end{array}
\right.
$$
It is a polynomial periodic function in $d$, and the general bound $EDdeg(F_{2,d}) \leq d^2$ is the best possible polynomial bound. Comparing with $MLdeg$, both are periodic while their periods are different. 

\end{exm}

The system for critical points of the distance function is given by 
\begin{equation} \label{sys1}
\left\{ \begin{array}{ll}
x_0^d + \cdots +x_n^d = 0\\
x_i^{d-1}(x_j - u_j) = x_j^{d-1}(x_i - u_i) \text{ for each }i\neq j
\end{array} \right.
\end{equation}
where the vector $u = (u_0, \dots, u_n) \in \mathbb{C}^n$ is sufficiently generic.
The ED-degree is the number of solutions of \eqref{sys1} except $(0,\dots,0)$, which is a (unique) singular point of the cone over the Fermat hypersurface $F_{n,d}$.

Introducing a new variable $t$, we modify the system \eqref{sys1} into following homogeneous system in $S[t] = \mathbb{C}[x_0,\dots,x_n,t]$.
\begin{equation} \label{sys2}
\left\{ \begin{array}{ll}
x_0^d + \cdots +x_n^d = 0\\
x_i^{d-1}(x_j - u_jt) = x_j^{d-1}(x_i - u_it) \text{ for each }i\neq j
\end{array} \right.
\end{equation}
Each solution of \eqref{sys2} of the form $(c_0 : \dots : c_n : 1)$ corresponds to the solution $(c_0, \dots, c_n)$ of \eqref{sys1}.
The system \eqref{sys2} has more solutions that we don't want to count.
Let $mult(0)$ be the multiplicity of $(0:\dots:0:1)$ for the system \eqref{sys2}, and $\epsilon(n,d)$ be the number of solutions of the form $(a_0:\dots:a_n:0)$ counting multiplicities. Then the ED-degree of $F_{n,d}$ is given by
\begin{equation} \label{edstr}
EDdeg(F_{n,d}) = deg(\eqref{sys2}) - mult(0) - \epsilon(n,d)
\end{equation}
where $deg(\eqref{sys2})$ is the degree of the projective scheme defined by the system \eqref{sys2}. Now, Theorem \ref{main} is just a consequence of following lemmas.

\begin{lemma} \label{mul}
The multiplicity of $(0:\dots:0:1)$ for the system \eqref{sys2}, denoted by $mult(0)$, is $d(d-1)^n$.
\end{lemma}
\begin{proof}
Let $I$ be the ideal in $S[t]$ generated by equations in \eqref{sys2} and $\mathfrak{m} = (x_0,\dots, x_n)$ be the ideal corresponding the point $(0:\dots:0:1)$. Then $mult(0)$ is defined by the length of $S[t]_\mathfrak{m}/I_\mathfrak{m}$ as an $S[t]_\mathfrak{m}$-module. In the local ring $S[t]_\mathfrak{m}$, the factor $(x_i-u_it)$ is a unit. (By the genericity of $u_i$, we may assume $u_i \neq 0$ for all $i$.) Writing $\mu_i = (x_0-u_0t)\cdot(x_i-u_it)^{-1}$, the localized ideal $I_\mathfrak{m}$ is generated by
\begin{equation} \label{mult}
\left\{ \begin{array}{ll}
x_0^d + \cdots +x_n^d = 0\\
x_0^{d-1} = x_i^{d-1}\mu_i \text{ for each }i.
\end{array} \right.
\end{equation}
Here, the length of $S[t]_\mathfrak{m}/I_\mathfrak{m}$ is just the maximum size of a monomial set in $S$ which are independent modulo $I_\mathfrak{m}$. By direct counting, we see that $mult(0) = |\{x_0^{\alpha_0} \cdot x_n^{\alpha_n} \mid 0\leq \alpha_0 \leq d-1, 0\leq \alpha_1,\dots,\alpha_n \leq d-2 \}| = d(d-1)^n$. Alternatively, it is same as $dim_\mathbb{C}(S/\bar{I})$ where $\bar{I}$ is the ideal in $S$ defined by \eqref{mult} after changing each $\mu_i$ into arbitrary nonzero value in $\mathbb{C}$. Therefore, by B\'{e}zout theorem, we get the same answer.    
\end{proof}

To compute $\epsilon(n,d)$ in \eqref{edstr}, we want to put $t=0$ in the system \eqref{sys2} to get
\begin{equation} \label{sys3}
\left\{ \begin{array}{ll}
x_0^d + \cdots +x_n^d = 0\\
x_i^{d-1}x_j = x_j^{d-1}x_i \text{ for each }i\neq j
\end{array} \right.
\end{equation}
This could give the wrong answer if \eqref{sys2} and the hyperplane $t=0$ meet non-transversally. The next lemma shows that it is not the case.

\begin{lemma} \label{trans}
The system \eqref{sys2} and the hyperplane $t=0$ meet transversally. Hence $\epsilon(n,d) = deg(\eqref{sys3})$.
\end{lemma}
\begin{proof}
It suffices to show that the system \eqref{sys2} and $\frac{\partial}{\partial t} \eqref{sys2}$ have no common root where $\frac{\partial}{\partial t} \eqref{sys2}$ is the ideal generated by the $t$-directional partial derivatives of the equations in \eqref{sys2}. Let $c=(c_0:\dots:c_n:0)$ be a nonzero solution of \eqref{sys2}. It has at least two nonzero entries by the equation $c_0^d+\cdots+c_n^d=0$. Without loss of generality, we may assume $c_0=1$ and $c_1 \neq 0$, which implies $c_1^{d-2} = 1$. Also,
\begin{align*}
&\frac{\partial}{\partial t} \left(x_0^{d-1}(x_1-u_1t)-x_1^{d-1}(x_0-u_0t)\right)|_{x=c}\\
&=-u_1+c_1^{d-1}u_0\\
&= 0 \text{ only if } |u_0| = |u_1|.
\end{align*}  
By the genericity of $u_i$, the last equality does not happen.
\end{proof}

\begin{lemma} \label{e-d}
Let $\delta(p,m)$ be the function defined in Theorem \ref{main}. Then we have
$$
\epsilon(n,d) = \sum_{m=1}^{n} \binom{n+1}{m+1} \cdot \delta(m,d-2).
$$
\end{lemma}
\begin{proof}
By Lemma \ref{trans}, $\epsilon(n,d)=deg(\eqref{sys3})$.
Let $(c_0:\cdots:c_n)$ be a solution of \eqref{sys3}. Suppose that $c_i \neq 0$ for all $i$. Then $x_0=1$ in the system \eqref{sys3} implies that all $c_i$'s are some $(d-2)$-nd roots of unity. Fix a $(d-2)$-nd primitive root of unity $\zeta$, and write $c_i = \zeta^t_i$. Then the system \eqref{sys3} has $\delta(n,d-2)$ many solutions. If a solution has $m+1$ many nonzero coordinates, the number of such solutions is $\binom{n+1}{m+1}$ (for the choices of nonzero coordinates) times $\delta(m,d-2)$.
\end{proof}

\begin{lemma} \label{degree}
The degree of the system \eqref{sys2} is given by
$$deg(\eqref{sys2}) = d\sum_{i=0}^{n} (d-1)^i.$$
\end{lemma}
\begin{proof}
Let $I$ be the ideal generated by
$$
x_i^{d-1}(x_j - u_jt) = x_j^{d-1}(x_i - u_it) \text{ for each }i\neq j 
$$
$$
=2 \times 2 \text{ minors of } 
\left[ \begin{array}{ccc}
x_0^{d-1} & \cdots & x_n^{d-1}\\
x_0-u_0t & \cdots & x_n-u_nt
\end{array} \right]
$$
It defines a curve in $\mathbb{P}_\mathbb{C}^{n+1}$. By Lemma \ref{trans}, $deg(I) = deg(I +(t))$. If $c=(c_0 : \dots : c_n : 0)$ is a solution for $I +(t)$, write $c_m = 1$ where $m$ is the first nonzero entry of $c$. Then for each nonzero entry except $c_m$, there are $d-2$ choices to be a solution. Hence the total number of solutions is
$$
\sum_{i=1}^{n+1} \binom{n+1}{i} (d-2)^{i-1}
 = \sum_{i=0}^{n} (d-1)^i.$$
 Therefore $deg(\eqref{sys2}) = deg(F_{n,d}) \cdot deg(I) = d\sum_{i=0}^{n} (d-1)^i$ by B\'{e}zout.
\end{proof}

\begin{proof}[Proof of Theorem \ref{main}]
We have
$$
EDdeg(F_{n,d}) = deg(\eqref{sys2}) - mult(0) - \epsilon(n,d).
$$
Apply Lemma \ref{degree}, \ref{mul}, and \ref{e-d} to each term in the right side.
\end{proof}

\begin{exm} \label{m2}
We showed that the ED-degrees of the Fermat hypersurfaces can be computed by $\delta(m,p)$ or $\epsilon(m,p)$ without using the random data $u$. The following Macaulay2 code computes the ED-degree of $F_{n,d}$ efficiently.

\begin{verbatim}
n=2,d=5;
R=QQ[x_0..x_n];
gbd=0;for i from 0 to n-1 do gbd=gbd+d*(d-1)^i;  -- the general bound
F=sum apply(n+1,i->(gens R)_i^2);
M=matrix{apply(n+1,i->((gens R)_i)^(d-1))}||matrix{gens R};
I=ideal(F)+minors(2,M);
EDdeg=gbd-(degree I)
\end{verbatim}
The output reveals that the Fermat quintic cone $F_{2,5}$ has ED-degree $23$.
\end{exm}

\subsection{Affine Fermat Hypersurfaces}
Let $X$ be the affine Fermat hypersurfaces $AF_{n,d}$. The system for critical points of the distance function is given by 
$$
\left\{ \begin{array}{ll}
x_1^d + \cdots +x_n^d = 1\\
x_i^{d-1}(x_j - u_j) = x_j^{d-1}(x_i - u_i) \text{ for each }i\neq j
\end{array} \right.
$$
and the homogenized system is
\begin{equation} \label{affsyshom}
\left\{ \begin{array}{ll}
x_1^d + \cdots +x_n^d = t^d\\
x_i^{d-1}(x_j - u_jt) = x_j^{d-1}(x_i - u_it) \text{ for each }i\neq j
\end{array} \right.
\end{equation}
In this case, $(0:\cdots:0:1)$ is not a solution for \eqref{affsyshom} (see Lemma \ref{mul}). Except that, the ED-degree of $AF_{n,d}$ can be computed in the same way as in the homogeneous cases.
\begin{cor} \label{affcase}
The ED-degree of the affine Fermat hypersurface $AF_{n,d}$ is given~by
$$
EDdeg(AF_{n,d}) = d\sum_{i=0}^{n-1} (d-1)^i - \sum_{m=1}^{n-1} \binom{n}{m+1} \cdot \delta(m,d-2).
$$
\end{cor}
Note that the summand is the general bound for affine varieties, given in \cite[Corollary 2.5]{DHOST}

\subsection{Real Critical Points}
For odd $d$, the Fermat hypersurface $F_{n,d}$ can be considered as a nonempty real variety. In this case, the number of the real critical points of the squared distance function highly depends on the location of the given point $u \in \mathbb{R}^{n+1}$. Nonetheless, the next theorem gives an upper bound for the maximum possible (finite) number of the real critical points 

\begin{thm} \label{realbd}
For the Fermat hypersurface $F_{n,d}$, the number of the nonzero real critical points of the squared distance function is bounded by
$$
\sqrt{2}^{25n^2-3n+2}\cdot(n+2)^{5n}
$$
\end{thm}
\begin{proof}
Let $u \in \mathbb{R}^{n+1}$ be a point not in $F_{n,d}$, whose entries are all nonzero. Then the critical equation \eqref{sys1} can be written by 
$$
\left\{ \begin{array}{ll}
x_0^d + \cdots +x_n^d = 0\\
x_i^{d-1}(x_0 - u_0) = x_0^{d-1}(x_i - u_i) \text{ for each }i=1,\dots,n
\end{array} \right.
$$
This system has $n+1$ polynomials in $n+1$ variables, and the number of monomials used in this system is $5n+1$, which does not depent on $d$. By Khovanskii's fewnomial bound \cite{Kho}, this system has at most
$$2^{\binom{5n}{2}}\cdot(n+2)^{5n}$$
positive solutions. It is also an upper bound for the number of real solutions in any orthant, hence we can have at most
$$2^{n+1}\cdot 2^{\binom{5n}{2}}\cdot(n+2)^{5n}$$ 
in total.
\end{proof}

Note that this bound does not depend on $d$, hence we can ask for the sharp bound for each $n$. For $n=1$, the real cone of $F_{1,d}$ is a straight line in $\mathbb{R}^2$, hence the critical equation has one real solution. For $n=2$, the maximum possible number seems to be $3$, but we don't have any proof for this and higher dimensional cases. Since the problem is highly related to the root of unity, we guess that each real solution produces many non-real roots in some way. In particular, the number of possible real solutions may not be more than linear.

\begin{conj} \label{con} 
The number of real critical points of \eqref{sys1} is at most $2n-1$.
\end{conj}

We note that Theorem \ref{realbd} is also valid for the scaled Fermat hypersurface $SF^a_{n,d}$ since the critical system contains the same number of monomials for all scaling vectors $a$.

\section{Scaled Fermat Hypersurfaces}
\subsection{Genericity of scaled Fermat hypersurfaces}

Recall the relation \eqref{edstr}
$$
F_{n,d} = deg(\eqref{sys2}) - mult(0) - \epsilon(n,d).
$$
The first two terms in this expression are invariant under any $GL(n+1,\mathbb{C})$ action (acting on the variables), thus we only focus on the last term $\epsilon(n,d)$, which is a sum of $\delta(m,d-2)$ with binomial coefficients (See Lemma \ref{e-d}).\\
For a given scaling vector $a\in (\mathbb{C}^*)^{m+1}$, define $\delta(m,p,a)$ to be the number of solutions of 
$$
\left\{ \begin{array}{ll}
1+x_1^2+ \cdots +x_m^2 = 0\\
x_i^p=a_i/a_0 \text{ for each }i=1,\cdots,m
\end{array} \right.
$$
whose entries are all nonzero. Note that $\delta(m,p,\mathbf{1})=\delta(m,p)$ where $\mathbf{1}=(1,\cdots,1)$. 
For $I \subseteq \{0,\dots,n\}$, let $a_I = (a_{i_0},\dots,a_{i_{|I|-1}})$ where $a_{i_j}$ is the $j$-th entry of $a$. Now the ED-degree of $SF_{n,d}^a$ is given by
$$EDdeg(SF_{n,d}^a) = d\sum_{i=0}^{n-1} (d-1)^i - \sum_{I \subseteq \{0,\dots,n\}} \delta(|I|-1,d-2,a_I).$$
Therefore the ED-degree of $SF_{n,d}^a$ achieves the equality in the general bound (Proposition \ref{gbound}) if and only if the latter summands are all zero. To examine, we need to define the {\em exponential cyclotomic polynomial} $Q_{m.p} \in \mathbb{Z}[x_0,\dots,x_m]$.

For an integer $p$ and a primitive $p$-th root of unity $\zeta$, consider the polynomial
$$
P_{m,p}(A_0,\dots,A_m)=\prod_{t_1,\dots,t_m =1}^{p} \left(A_0+\sum_{k=1}^m\zeta^{t_k} A_i\right).
$$
One can easily see that $P_{m,p}(A_0,\dots,A_m) \in \mathbb{Z}[A_0^p,\dots,A_m^p]$. Replace $A_i^p$ by $x_i$ to get a polynomial $Q_{m,p}(x_0,\dots,x_m)$, i.e, $Q_{m,p}$ is the unique polynomial such that $Q_{m,p}(A_0^p,\dots, A_m^p) = P_{m,p}(A_0,\dots,A_m)$.

\begin{thm} \label{resultant}
$\delta(m,p,a)\neq0$ if and only if 
$$
\left\{ \begin{array}{ll}
Q_{m,p}(a_0^2,\dots,a_m^2)=0 & \text{for } p \text{ odd,}\\
Q_{m,p/2}(a_0,\dots,a_m)=0 & \text{for } p \text{ even.}
\end{array}
\right.
$$
\end{thm}

\begin{proof}
Let $a \in (\mathbb{C}^*)^{n+1}$ be given. We may assume $a_0 = 1$. For each $i$, choose a complex number $b_i$ so that  $b_i^p= a_i$. Let $\zeta$ be a primitive $p$-th root of unity.
By definition, $\delta(m,p,a)\neq0$ if and only if the system
$$
\left\{ \begin{array}{ll}
1+x_1^2+ \cdots +x_m^2 = 0\\
x_i^p=a_i \text{ for each }i=1,\cdots,m
\end{array} \right.
$$
has a solution whose entries are all nonzero. Any solution of the second equations is of the form
$(b_1 \zeta^{t_1}, \dots, b_m \zeta^{t_m})$. It satisfies the first equation if and only if $1 +  b_1^2 \zeta^{2t_1}+ \cdots + b_m^2 \zeta^{2t_m} = 0$. 
The image of the square map $z \mapsto z^2$ defined on the set of all $p$-th roots of unity is itself if $p$ is odd, or is the set of all $(p/2)$-nd roots of unity if $p$ is even. 
Therefore $\delta(m,p,a)\neq0$ if and only if
$$
\prod_{t_1,\dots,t_m=1}^{p} (1 +  b_1^2 \zeta^{t_1}+ \cdots + b_m^2 \zeta^{t_m})=0
$$ 
for $d$ odd,
$$
\prod_{t_1,\dots,t_m=1}^{p/2} (1 +  b_1^2 (\zeta^2)^{t_1}+ \cdots + b_m^2 (\zeta^2)^{t_m})=0
$$
for even $d$.
Now the theorem follows by replacing $b_i^p$ with $a_i$ after expanding the product. 
\end{proof}
\begin{cor}
For generic $a \in (\mathbb{C}^*)^{n+1}$,
$$
EDdeg(SF_{n,d}^a) = d\sum_{i=0}^{n-1} (d-1)^i 
$$
\end{cor}

The exponential cyclotomic polynomial $Q_{m,p}$ would be interesting itself. We close this section with a theorem showing that $Q_{m,p}$ has a nice property as an algebraic object.
\begin{thm}
For any integer $m$ and $p$, the exponential cyclotomic polynomial $Q_{m,p}$ is irreducible over $\mathbb{C}$.
\end{thm}
\begin{proof}
Let $f \in \mathbb{C}[x_0,\dots,x_m]$ be an irreducible factor of $Q_{m,p}$. Then $f(A_0^p,\dots,A_m^p)$ is a factor of $P_{m,p}(A_0,\dots,A_m)$, and we may assume that $f(A_0^p,\dots,A_m^p)$ is divisible by $(A_0+\cdots+A_m)$. For any $(t_1,\dots,t_m), t_i=1,\dots,p$, the polynomial $f(A_0^p,\dots,A_m^p)$ is stable under the action $A_i \rightarrow \zeta^{t_i}A_i$. Therefore $f(A_0^p,\dots,A_m^p)$ is divisible by $(A_0+\sum_i \zeta^{t_i}A_i)$. Since $(t_1,\dots,t_m)$ was arbitrary, $f(A_0^p,\dots,A_m^p)$ is divisible by every possible linear factor of $P_{m,p}(A_0,\dots,A_m)$. Therefore $f(A_0^p,\dots,A_m^p) = P_{m,p}(A_0,\dots,A_m)$ and hence $Q_{m,p} = f(x_0,\dots,x_m)$ up to scalar multiplication.
\end{proof}


\begin{thebibliography}{99}

\bibitem{betti} S.~Adams and P.~Sarnak: {\em Betti Numbers of Congruence Groups}, Israel J. of Math. \textbf{88} (1994), 31-72.
\bibitem{mldeg} D.~Agostini, D.~Alberelli, F.~Grande, and P.~Lella: {\em The maximum likelihood degree of fermat hypersurfaces}, {\tt arXiv:1404.5745}.
\bibitem{CLO} D.~Cox, J.~Little, and D.~O'Shea: {\em Ideals, Varieties, and Algorithms. An Introdution to Computational Algebraic Geometry and Commutative Algebra}, {Undergraduate Texts in Mathematics, Springer-Verlag, New York, 1992.}
\bibitem{DHOST} J.~Draisma, E.~Horobe\c{t}, G.~Ottaviani, B.~Sturmfels, and R.~Thomas: {\em The Euclidean Distance Degree of an Algebraic Variety}, {\tt arXiv:1309.0049}.
\bibitem{SH} J.~Huh, B.~Sturmfels: {\em Likelihood Geometry}, in Combinatorial Algebraic Geometry (eds. A.~Conca et al.), Lecture Note in Mathematics 2014, Springer (2014) 63-117
\bibitem{Kho} A.G.~Khovanskii: {\em A class of systems of transcendental equations}, Dokl. Akad. Nauk. SSSR \textbf{255} (1980), no. 4, 804-807.
\bibitem{Lau} M.~Laurent: {\em Equations diophantiennes exponentielles}, Invent. Math. \textbf{78} (1984), 299-327.
\bibitem{Mann} H.~Mann, {\em On linear relations between roots of unity}, Mathematika. \textbf{12} (1965), 107-117.
\bibitem{OSS} G.~Ottaviani, P.-J.~Spaenlehauer, and B.~Sturmfels: {\em Exact Solutions in Structured Low-Rank Approximation}, {\tt arXiv:1311.2376}
\end{thebibliography}
\end{document}